\documentclass[11pt]{article} 
\usepackage[a4paper]{geometry}

\usepackage{amsmath,amsthm,amssymb,amsfonts}
\usepackage{mathrsfs}
\usepackage{color,xcolor}
\usepackage{graphicx}
\usepackage{galois}
\usepackage[colorlinks,linkcolor=blue,citecolor=blue]{hyperref} %OPTIONAL
\usepackage{enumerate}
\usepackage{cite}
\graphicspath{{Figures/}} % Set the directory where figures are saved.

\newtheorem{Th}{Theorem}[section]
\newtheorem{Prop}[Th]{Proposition}
\newtheorem{Lem}[Th]{Lemma}
\newtheorem{Coro}[Th]{Corollary}
\newtheorem{Rem}[Th]{Remark}

\numberwithin{equation}{section}
\newcommand{\Real}{\mathbb{R}}
\newcommand{\Gmf}{\Gamma\hskip-1mm}
\newcommand{\Tran}[1]{#1^\mathrm{T}}
\newcommand{\Tr}{\mathrm{Tr}}
\newcommand{\Qfs}{\mathcal{Q}}
\newcommand{\Fun}{\mathcal{F}}

\newcommand{\comps}{\!\comp\!}
\newcommand{\nxn}[0]{{n\times n}}
\newcommand{\ball}[0]{\mathcal{B}}

\newcommand{\md}{\,\mathrm{d}}
\newcommand{\me}{\mathrm{e}}
\newcommand{\ellp}{\ell_p}
\newcommand{\ellt}{\ell_2}

\newcommand{\Sob}[2]{H^#2}

\def\clap#1{\hbox to 0pt{\hss#1\hss}}

\title{Notes on the Sobolev (Semi)Norms  of Quadratic Functions}
\author{Zaikun ZHANG\thanks{Institute of Computational Mathematics and Scientific/Engineering Computation, Chinese Academy of Sciences, P.O. Box 2719, Beijing 100190, CHINA ({\tt zhangzk@lsec.cc.ac.cn}).}}
\date{\today}
\begin{document}
\maketitle
\begin{abstract}
  This paper studies the $\Sob{2}{0}$ norm and $\Sob{2}{1}$ seminorm
  of quadratic functions. The~(semi)norms are expressed explicitly in
  terms of the coefficients of the quadratic function under
  consideration when the underlying domain is an $\ellp$-ball~($1\leq p
  \leq\infty$) in $\Real^n$. 

\end{abstract}

\section{Motivation and Introduction}
Zhang \cite{SobolevDFO} studies the $\Sob{2}{1}$ sminorm of quadratic
functions over $\ellt$-balls of $\Real^n$, and applies it to
derivative-free optimization problems. This
paper will investigate the $\Sob{2}{0}$ norm and $\Sob{2}{1}$ seminorm
of quadratic functions over $\ellp$-balls \mbox{($1\leq p
\leq\infty$)} of $\Real^n$. These (semi)norms may be useful for derivative-free optimization when considering trust region methods with $\ellp$ trust region.

This paper is organized as follows. Section \ref{notation} describes
the notations.~Section \ref{formula} presents the formulae for the Sobolev (semi)norms of quadratic functions over $\ellp$-balls. The formulae are proved in Section \ref{proofs}. Section \ref{diss} contains some discussions. Some propositions about the Gamma function are given in Appendix for reference. 

\section{Notation}\label{notation}
In this paper,~the following symbols will be used unless otherwise specified.
\begin{itemize}
  \item $Q$ is a quadratic function defined by 
    \begin{equation}
    Q(x) = \frac{1}{2}\Tran{x}Bx+\Tran{g}x+c,~~x\in \Real^n,
    \end{equation}
    where $B\in\Real^{\nxn}$ is a symmetric matrix,~$g\in\Real^n$ is a vector and~$c \in\Real$ is a scalar.
  \item $D\in\Real^{\nxn}$~is the diagonal of $B$.
\item $p\geq 1$ is a positive constant or $p=\infty$.
\item $\|\cdot\|_p$ is the $\ellp$-norm~on $\Real^n$,~and $\ball_p^r$ is the $\ellp$-ball centering $0$ with radius $r>0$,~i.e.,
  \begin{equation}
\ball_p^r = \{x\in\Real^n;~\|x\|_p \leq r\}.
  \end{equation}
Besides,~$V_{p}$ is the volume of $\ball_p^1$.
\item Given $x\in\Real^n$~and~$i\in\{1,~2,~...,~n\}$,~$x_i$ is the $i$-th coordinate of $x$.
\item $\Gamma(\cdot)$~is the Gamma function,~and~$\beta(\cdot,\cdot)$ is the Beta function \cite{artin-gamma}.
\item $\gamma$~is the Euler constant.
\item The following symbols are used for the Sobolev (semi)norms \cite{EvansPDE} of a function $f$ over a domain $\Omega\in\Real^n$:
\begin{align}
\|f\|_{\Sob{2}{0}(\Omega)}&\equiv\left[\int_{\Omega}|f(x)|^2\md x\right]^{1/2},\\
|f|_{\Sob{2}{1}(\Omega)}&\equiv\left[\int_{\Omega}\|\nabla f(x)\|_2^2\md x\right]^{1/2}.
\end{align}
%\item $\Gamma(\cdot)$~is Gamma function,~i.e.,
%$$\Gamma(z) = \int_0^{\infty}t^{z-1}e^{-t}\md t,~~z>0;$$
%$\beta(\cdot,\cdot)$ is Beta function,~i.e.,
%$$\beta(u,v) = \int_0^1t^{u-1}(1-t)^{v-1}\md t,~~u>0,~v>0.$$
\end{itemize}
\section{Formulae for the Sobolev (Semi)Norms}\label{formula}
The $\Sob{2}{0}$ norm and $\Sob{2}{1}$ seminorm of $Q$ over
$\ball_p^r$ can be expressed explicitly in terms of its coefficients
$B$,~$g$~and~$c$. We present the formulae in Theorem \ref{Thp}.~Notice that the case $p=\infty$ can also be covered by these formulae,~as will be pointed out in Remark \ref{inftycase}. The proof of Theorem \ref{Thp} is in Section \ref{proofs}. 

\begin{Th}\label{Thp}
\begin{equation}
\|Q\|^2_{\Sob{2}{0}(\ball_p^r)} =\frac{I}{2}\|B\|_\mathrm{F}^2+\frac{I}{4}\Tr^2B+\frac{J-3I}{4}\|D\|_\mathrm{F}^2+K(c\Tr B+\|g\|_2^2)+Vc^2,\\
\end{equation}
and 
\begin{equation}
|Q|^2_{\Sob{2}{1}(\ball_p^r)} =K\|B\|_\mathrm{F}^2+V\|g\|^2,
\end{equation}
%where%,~for $p\in[1,\infty)$, 
%\begin{align}
%\label{Ip}I&=\frac{[\Gamma(\frac{3}{p})]^2\Gamma(\frac{n}{p})}{[\Gamma(\frac{1}{p})]^2\Gamma(\frac{n+4}{p})}\cdot\frac{n}{n+4}\cdot V_p\,r^{n+4},\\
%\label{Jp}J&=\frac{\Gamma(\frac{5}{p})\Gamma(\frac{n}{p})}{\Gamma(\frac{1}{p})\Gamma(\frac{n+4}{p})}\cdot\frac{n}{n+4}\cdot V_p\,r^{n+4},\\
%\label{Kp}K&=\frac{\Gamma(\frac{3}{p})\Gamma(\frac{n}{p})}{\Gamma(\frac{1}{p})\Gamma(\frac{n+2}{p})}\cdot\frac{n}{n+2}\cdot V_p\,r^{n+2},\\
%V&=V_p\,r^n
%\end{align}
%\begin{equation}
%\begin{split}
%\|Q\|^2_{\Sob{2}{0}(\ball_p^r)} =V_p\,r^n\left[r^4\left(\frac{I_p}{2}\|B\|_\mathrm{F}^2+\frac{I_p}{4}\Tr^2B+\frac{J_p-3I_p}{4}\|D\|_\mathrm{F}^2\right)+r^2K_p(c\Tr B+\|g\|_2^2)+c^2\right],\\
%\|Q\|^2_{\Sob{2}{0}(\ball_p^r)} =V_p\,r^n\left[\frac{r^4}{4}\left({2I_p}\|B\|_\mathrm{F}^2+{I_p}\Tr^2B+(J_p-3I_p)\|D\|_\mathrm{F}^2\right)+r^2K_p(c\Tr B+\|g\|_2^2)+c^2\right],\\
%&\|Q\|^2_{\Sob{2}{0}(\ball_p^r)} \\
%=&V_p\,r^n\left[\frac{r^4}{4}\left({2I_p}\|B\|_\mathrm{F}^2+{I_p}\Tr^2B+(J_p-3I_p)\|D\|_\mathrm{F}^2\right)+r^2K_p(c\Tr B+\|g\|_2^2)+c^2\right],\\
%\|Q\|^2_{\Sob{2}{0}(\ball_p^r)} =V_p\,r^n\left[\frac{r^4}{4}\|B\|_{\ball_p}^2+r^2K_p(c\Tr B+\|g\|_2^2)+c^2\right],
%\end{split}
%\end{equation}
%and 
%\begin{equation}
%|Q|^2_{\Sob{2}{1}(\ball_p^r)} =K\|B\|_\mathrm{F}^2+V\|g\|^2,
%\end{equation}
where%,~for $p\in[1,\infty)$, 
\begin{align}
\label{Ip}I&=\frac{[\Gamma(\frac{3}{p})]^2\Gamma(\frac{n}{p})}{[\Gamma(\frac{1}{p})]^2\Gamma(\frac{n+4}{p})}\cdot\frac{n}{n+4}\cdot V_p\,r^{n+4},\\
\label{Jp}J&=\frac{\Gamma(\frac{5}{p})\Gamma(\frac{n}{p})}{\Gamma(\frac{1}{p})\Gamma(\frac{n+4}{p})}\cdot\frac{n}{n+4}\cdot V_p\,r^{n+4},\\
\label{Kp}K&=\frac{\Gamma(\frac{3}{p})\Gamma(\frac{n}{p})}{\Gamma(\frac{1}{p})\Gamma(\frac{n+2}{p})}\cdot\frac{n}{n+2}\cdot V_p\,r^{n+2},\\
V&=V_p\,r^n.
\end{align}
\end{Th}
\begin{Rem}\label{inftycase}
Theorem \ref{Thp} covers the case $p=\infty$ as well.~When $p=\infty$,~we interpret~(\ref{Ip}~--~\ref{Kp})~in the sense of limit,~i.e.,~
%\begin{align}
%\addtocounter{equation}{-3}
%\tag{\theequation$'$}
%I_\infty&=\lim_{p\rightarrow\infty}\frac{[\Gamma(\frac{3}{p})]^2\Gamma(\frac{n}{p})}{[\Gamma(\frac{1}{p})]^2\Gamma(\frac{n+4}{p})}\cdot\frac{n}{n+4}V_\infty,\\
%\addtocounter{equation}{1}
%\tag{\theequation$'$}
%J_\infty&=\lim_{p\rightarrow\infty}\frac{\Gamma(\frac{5}{p})\Gamma(\frac{n}{p})}{\Gamma(\frac{1}{p})\Gamma(\frac{n+4}{p})}\cdot\frac{n}{n+4}V_\infty,\\
%\addtocounter{equation}{1}
%\tag{\theequation$'$}
%K_\infty&=\lim_{p\rightarrow\infty}\frac{\Gamma(\frac{3}{p})\Gamma(\frac{n}{p})}{\Gamma(\frac{1}{p})\Gamma(\frac{n+2}{p})}\cdot\frac{n}{n+2}V_\infty.
%\end{align}
\begin{align}
\addtocounter{equation}{-3}
\tag{\theequation$'$}
I&=\lim_{p\rightarrow\infty}\frac{[\Gamma(\frac{3}{p})]^2\Gamma(\frac{n}{p})}{[\Gamma(\frac{1}{p})]^2\Gamma(\frac{n+4}{p})}\cdot\frac{n}{n+4}\cdot V_\infty\,r^{n+4},\\
\addtocounter{equation}{1}
\tag{\theequation$'$}
J&=\lim_{p\rightarrow\infty}\frac{\Gamma(\frac{5}{p})\Gamma(\frac{n}{p})}{\Gamma(\frac{1}{p})\Gamma(\frac{n+4}{p})}\cdot\frac{n}{n+4}\cdot V_\infty\,r^{n+4},\\
\addtocounter{equation}{1}
\tag{\theequation$'$}
K&=\lim_{p\rightarrow\infty}\frac{\Gamma(\frac{3}{p})\Gamma(\frac{n}{p})}{\Gamma(\frac{1}{p})\Gamma(\frac{n+2}{p})}\cdot\frac{n}{n+2}\cdot V_\infty\,r^{n+2}.
\end{align}
%and $I_\infty$,~$J_{\infty}$ and $K_\infty$ are the limits of $I_p$,~$J_p$~and~$K_p$ when $p$ tends to $\infty$.
The limits above can be calculated easily with the help of proposition \ref{limit}.
\end{Rem}
As illustrations of Theorem \ref{Thp},~we present the (semi)norms with $p=2$ as follows.
%\begin{Coro}\label{Th1}
%\begin{align}
%\|Q\|^2_{\Sob{2}{0}(\ball_1^r)} &= V_{1}r^n\left[\frac{n!\,r^4}{(n+4)!}\left(2\|B\|_\mathrm{F}^2+\Tr^2 B+3\|D\|_\mathrm{F}^2\right)+\frac{2n!\,r^2}{(n+2)!}(c\Tr B+\|g\|_2^2)+c^2\right];\\
%|Q|^2_{\Sob{2}{1}(\ball_1^r)} &= V_{1}r^n\left[\frac{2n!\,r^2}{(n+2)!}\|B\|_\mathrm{F}^2+\|g\|_2^2\right].
%\end{align}
%\end{Coro}
\begin{Coro}\label{Th2}
\begin{align}
\|Q\|^2_{\Sob{2}{0}(\ball_2^r)} &= V_{2}\,r^n\!\left[\frac{r^4\left(2\|B\|_\mathrm{F}^2+\Tr^2 B\right)}{4(n+2)(n+4)}+\frac{r^2(c\Tr B+\|g\|_2^2)}{n+2}+c^2\right];\\
|Q|^2_{\Sob{2}{1}(\ball_2^r)} &=
V_{2}\,r^n\!\left[\frac{r^2}{n+2}\|B\|_\mathrm{F}^2+\|g\|_2^2\right].
\label{H0SNL2}
\end{align}
\end{Coro}

Notice that the formula (\ref{H0SNL2}) has been proved in Zhang
\cite{SobolevDFO}.
%\begin{Coro}\label{Thi}
%\begin{align}
%\|Q\|^2_{\Sob{2}{0}(\ball_\infty^r)} &= V_{\infty}\,r^n\!\left[\frac{r^4}{180}\left(10\|B\|_\mathrm{F}^2+5\Tr^2 B-6\|D\|_\mathrm{F}^2\right)+\frac{r^2}{3}(c\Tr B+\|g\|_2^2)+c^2\right];\\
%|Q|^2_{\Sob{2}{1}(\ball_\infty^r)} &= V_{\infty}\,r^n\!\left[\frac{r^2}{3}\|B\|_\mathrm{F}^2+\|g\|_2^2\right].
%\end{align}
%\end{Coro}

\section{Proofs of Main Results}\label{proofs}
We assume $n\geq 2$ henceforth,~because everything is trivial when $n=1$.
\subsection{Lemmas}
%For simplicity,~we give some lemmas.%~First is a proposition about the volume of unit $\ellp$-balls~($p=1$,~$2$,~$\infty$),~which will be used in the proof of Lemma~\ref{int}.
%\begin{Lem}
%Suppose $n\geq 3$, and denote the volume of unit $\ellp$-ball in $\Real^{n-2}$ by $V_{p,n-2}$. We have
%\begin{align}
%\label{v1}\frac{V_{1,n}}{V_{1,n-2}} &= \frac{4}{n(n-1)};\\
%\label{v2}\frac{V_{2,n}}{V_{2,n-2}} &= \frac{2\pi}{n};\\
%\label{v3}\frac{V_{\infty,n}}{V_{\infty,n-2}} &= 4.
%\end{align}
%\end{Lem}
%\begin{proof}
%(\ref{v1})~and~(\ref{v3})~are trivial.~(\ref{v2})~is also easy with the help of Beta function
%$$\beta(u,v) = \int_0^1t^{u-1}(1-t)^{v-1}\md t,~~u>0,~v>0.$$
%\end{proof}

For Simplicity, we first prove some lemmas.

First we investigate the integrals of some monomials,~which will be presented in the following two lemmas.
\begin{Lem}\label{null}
Suppose~$1\leq i, j, k, l\leq n$.
\begin{itemize}
\item[a.]The integrals of $x_i$~and~$x_ix_jx_k$~over~$\ball_p^r$~are $0$.
\item[b.]The integrals of $x_ix_j$~and~$x_ix_jx_k^2$~over~$\ball_p^r$~are $0$,~provided that $i\neq j$.
\item[c.]The integral of $x_ix_jx_kx_l$~over~$\ball_p^r$~is $0$,~provided that $i$,~$j$~and~$k$~are pairwise different.
\end{itemize}
\end{Lem}

Lemma \ref{null} is trivial so we omit the proof.

\begin{Lem}\label{int}
Suppose~$1\leq i < j \leq n$,~and $k_1$, $k_2$ are even natural numbers.~Then 
\begin{equation}
\label{intp}
\int_{\ball_p^r}x_i^{k_1}x_j^{k_2}\md x= \frac{\Gamma(\frac{k_1+1}{p})\Gamma(\frac{k_2+1}{p})\Gamma(\frac{n}{p})}{\Gamma(\frac{1}{p})\Gamma(\frac{1}{p})\Gamma(\frac{n+k_1+k_2}{p})}\cdot\frac{n}{n+k_1+k_2}\cdot V_p\,r^{n+k_1+k_2}.
\end{equation}
%\begin{itemize}
%\item[a.]
%\begin{align}
%\int_{\ball_1^r}x_i^2x_j^2\md x&= \frac{4n!\,V_{1,n}r^{n+4}}{(n+4)!};\\
%\int_{\ball_1^r}x_i^4\md x&= \frac{24n!\,V_{1,n}r^{n+4}}{(n+4)!};\\
%\int_{\ball_1^r}x_i^2\md x&= \frac{2n!\,V_{1,n}r^{n+2}}{(n+2)!}.
%\end{align}
%\item[b.]
%\begin{align}
%\label{int1}\int_{\ball_2^r}x_i^2x_j^2\md x&= \frac{V_{2,n}r^{n+4}}{(n+2)(n+4)};\\
%\int_{\ball_2^r}x_i^4\md x&= \frac{3V_{2,n}r^{n+4}}{(n+2)(n+4)};\\
%\int_{\ball_2^r}x_i^2\md x&= \frac{V_{2,n}r^{n+2}}{n+2}.
%\end{align}
%\item[c.]
%\begin{align}
%\int_{\ball_\infty^r}x_i^2x_j^2\md x&= \frac{V_{\infty,n}r^{n+4}}{9};\\
%\int_{\ball_\infty^r}x_i^4\md x&= \frac{V_{\infty,n}r^{n+4}}{5};\\
%\int_{\ball_\infty^r}x_i^2\md x&= \frac{V_{\infty,n}r^{n+2}}{3}.
%\end{align}
%\end{itemize}
\end{Lem}
\begin{Rem}
Lemma \ref{int} covers the case $p=\infty$ as well.~When $p=\infty$,~we interpret~(\ref{intp})~in the sense of limit,~i.e.,~
\begin{equation}
\tag{\theequation$'$}\label{inti}
\int_{\ball_\infty^r}\!\! x_i^{k_1}x_j^{k_2}\md x= \lim_{p\rightarrow \infty}\frac{\Gamma(\frac{k_1+1}{p})\Gamma(\frac{k_2+1}{p})\Gamma(\frac{n}{p})}{\Gamma(\frac{1}{p})\Gamma(\frac{1}{p})\Gamma(\frac{n+k_1+k_2}{p})}\cdot\frac{n}{n+k_1+k_2}\cdot V_\infty\,r^{n+k_1+k_2}.
\end{equation}
\end{Rem}
\begin{proof}
Without loss of generality,~we assume $r=1$.~Suppose additionally $n\geq 3$~since things are trivial if $n=2$.~Denote the volume of the unit $\ellp$-ball in $\Real^{n-2}$ by $V_{p,n-2}$.

We fist justify~(\ref{intp})~for~$p\in[1,\infty)$,~and then show its validity for $p=\infty$ in the sense of~(\ref{inti}).

When $p\in[1,\infty)$,
\begin{align}
&\int_{B_p^1}x_i^{k_1}x_j^{k_2}\md x\\
=~&\int_{|u|^{p}+|v|^{p}\leq1}u^{k_1}v^{k_2}\md u\md v\int_{w\in\Real^{n-2},\|w\|_p\leq(1-|u|^p-|v|^p)^{\frac{1}{p}}}\md w\\
=~&V_{p,n-2}\int_{|u|^p+|v|^p\leq 1}u^{k_1}v^{k_2}(1-|u|^p-|v|^p)^{\frac{n-2}{p}}\md u\md v\\
=~&4V_{p,n-2}\int_{u^p+v^p\leq 1,~u,v\geq 0}u^{k_1}v^{k_2}(1-u^p-v^p)^{\frac{n-2}{p}}\md u\md v.
\end{align}
Consider the transformation
\begin{align}
u^{\frac{p}{2}}&=\rho\cos\theta,\\
v^{\frac{p}{2}}&=\rho\sin\theta.
\end{align}
Then we have
%=~&4V_{p,n-2}\int_0^{\frac{\pi}{2}}\md\theta\int_0^1 (\rho\cos\theta)^{\frac{2k_1}{p}}(\rho\sin\theta)^{\frac{2k_2}{p}}(1-\rho^2)^{\frac{n-2}{p}}\cdot\frac{4}{p^2}\rho^{\frac{4}{p}-1}(\cos\theta\sin\theta)^{\frac{2}{p}-1}\md \rho\\
\begin{align}
&\int_{B_p^1}x_i^{k_1}x_j^{k_2}\md x\\
=~&\frac{16}{p^2}V_{p,n-2}\int_0^{\frac{\pi}{2}}(\cos\theta)^{\frac{2k_1+2}{p}-1}(\sin\theta)^{\frac{2k_2+2}{p}-1}\md\theta\int_0^1\rho^{\frac{2k_1+2k_2+4}{p}-1}(1-\rho^2)^{\frac{n-2}{p}}\md\rho\\
=~&\frac{4}{p^2}V_{p,n-2}\,\beta\!\left(\frac{k_1+1}{p},\frac{k_2+1}{p}\right)\beta\!\left(\frac{k_1+k_2+2}{p},\frac{n-2}{p}+1\right).
\end{align}
By setting $k_1$ and $k_2$ to $0$,~we obtain
\begin{equation}
V_p=\frac{4}{p^2}V_{p,n-2}\,\beta\!\left(\frac{1}{p},\frac{1}{p}\right)\beta\!\left(\frac{2}{p},\frac{n-2}{p}+1\right).
\end{equation}
Hence
\begin{align}
&\int_{B_p^1}x_i^{k_1}x_j^{k_2}\md x\\
=~&\frac{\beta\!\left(\frac{k_1+1}{p},\frac{k_2+1}{p}\right)\beta\!\left(\frac{k_1+k_2+2}{p},\frac{n-2}{p}+1\right)}{\beta\!\left(\frac{1}{p},\frac{1}{p}\right)\beta\!\left(\frac{2}{p},\frac{n-2}{p}+1\right)}V_p\\
=~&\frac{\Gmf\left(\frac{k_1+1}{p}\right)\Gmf\left(\frac{k_2+1}{p}\right)\Gmf\left(\frac{k_1+k_2+2}{p}\right)\Gmf\left(\frac{n-2}{p}+1\right)\Gmf\left(\frac{2}{p}\right)\Gmf\left(\frac{n}{p}+1\right)}{\Gmf\left(\frac{1}{p}\right)\Gmf\left(\frac{1}{p}\right)\Gmf\left(\frac{2}{p}\right)\Gmf\left(\frac{n-2}{p}+1\right)\Gmf\left(\frac{k_1+k_2+2}{p}\right)\Gmf\left(\frac{n+k_1+k_2}{p}+1\right)}V_p\\
%=~&\frac{\Gmf\left(\frac{k_1+1}{p}\right)\Gmf\left(\frac{k_2+1}{p}\right)\Gmf\left(\frac{n}{p}+1\right)}{\Gmf\left(\frac{1}{p}\right)\Gmf\left(\frac{1}{p}\right)\Gmf\left(\frac{n+k_1+k_2}{p}+1\right)}V_p\\
=~&\frac{\Gmf\left(\frac{k_1+1}{p}\right)\Gmf\left(\frac{k_2+1}{p}\right)\Gmf\left(\frac{n}{p}\right)}{\Gmf\left(\frac{1}{p}\right)\Gmf\left(\frac{1}{p}\right)\Gmf\left(\frac{n+k_1+k_2}{p}\right)}\cdot\frac{n}{n+k_1+k_2}V_p.
\end{align}
Thus~(\ref{intp})~holds for finite $p$.

Now consider infinite $p$.~According to Lebesgue's Dominated Convergence Theorem,~we have
%\begin{equation}\label{Lebesgue}
%\int_{\ball_p^1}\!x_i^{k_1}x_j^{k_2}\md x\rightarrow\!\!\int_{\ball_\infty^1}\!\!x_i^{k_1}x_j^{k_2}\md x~~\textnormal{and}~~~V_p\rightarrow V_{\infty}~~~(p\rightarrow\infty).
%\end{equation}
\begin{equation}
\int_{\ball_p^1}\!x_i^{k_1}x_j^{k_2}\md x\rightarrow\!\!\int_{\ball_\infty^1}\!\!x_i^{k_1}x_j^{k_2}\md x~~\textnormal{and}~~~V_p\rightarrow V_{\infty}~~~(p\rightarrow\infty).
\end{equation}
which together with~(\ref{intp})~implies the convergence of
$\frac{\Gamma\left(\frac{k_1+1}{p}\right)\Gamma\left(\frac{k_2+1}{p}\right)\Gamma\left(\frac{n}{p}
\right)}{\Gamma\left(\frac{1}{p}\right)\Gamma\left(\frac{1}{p}\right)\Gamma\left(\frac{n+k_1+k_2}{p}\right)}$
when $p$ tends to infinity and the validity of~(\ref{inti}).~Thus~(\ref{intp})~holds for infinite $p$ in the sense of limit.
\end{proof}
\begin{Rem}
Via straightforward calculus,~we can show that
\begin{equation}
\int_{\ball_{\infty}^r}\!\! x_i^{k_1}x_j^{k_2}\md x=\frac{V_{\infty}\,r^{n+k_1+k_2}}{(k_1+1)(k_2+1)},
\end{equation}
which is the same with~(\ref{inti})~according to Proposition \ref{limit}.
\end{Rem}
%In light of Lemma \ref{null},~the following proposition is straightforward. 
Now consider some more integrals which will be used in the computation of the (semi)norms. 
\begin{Lem}\label{integrals}
Denote $\int_{\ball_p^r}x_1^2x_2^2\md x$,~$\int_{\ball_p^r}x_1^4\md x$ and $\int_{\ball_p^r}x_1^2\md x$ by $I$,~$J$~and~$K$.~Then we have
\begin{itemize}
\item[a.]
\begin{align}
\label{easy1}\int_{\ball_p^r}&\Tran{g}x\md x= 0;\\
\int_{\ball_p^r}&\Tran{g}Bx\md x= 0;\\
\label{easy2}\int_{\ball_p^r}&(\Tran{g}x)(\Tran{x}Bx)\md x= 0.
\end{align}
\item[b.]
\begin{align}
\label{dif1}\int_{\ball_p^r}\Tran{x}Bx\md x&= K\Tr B;\\
\label{dif2}\int_{\ball_p^r}\Tran{x}B^2x\md x&= K\|B\|_\mathrm{F}^2;\\
\label{diff}\int_{\ball_p^r}(\Tran{x}Bx)^2\md x&= I(2\|B\|_\mathrm{F}^2+\Tr^2B)+(J-3I)\|D\|_\mathrm{F}^2;\\
\label{dif3}\int_{\ball_p^r}(\Tran{g}x)^2\md x&= K\|g\|_2^2.
\end{align}
\end{itemize}
\end{Lem}
\begin{proof}
(\ref{easy1}~--~\ref{easy2})~follow directly from Lemma \ref{null}.~As for~(\ref{dif1}~--~\ref{dif3}),~we only verify~(\ref{diff})~as an example,~because the others are similar and much easier.

Denote the~($i, j$)~entry of $B$ by $B_{ij}$.~According to Lemma \ref{null},
\begin{equation}
\begin{split}
&~\int_{\ball_p^r}(\Tran{x}Bx)^2\md x\\
=&~\int_{\ball_p^r}\sum_{i,j,k,l}(x_iB_{ij}x_j)(x_kB_{kl}x_l)\md x\\
=&~\int_{\ball_p^r}\left(\sum_{i}B^2_{ii}x_i^4+2\sum_{i\neq j}B^2_{ij}x_i^2x_j^2+\sum_{i\neq j}B_{ii}B_{jj}x_i^2x_j^2\right)\md x\\
=&~J\|D\|_\mathrm{F}^2+2I(\|B\|^2_\mathrm{F}-\|D\|^2_\mathrm{F})+I(\Tr^2B-\|D\|_\mathrm{F}^2)\\
=&~I(2\|B\|_\mathrm{F}^2+\Tr^2B)+(J-3I)\|D\|_\mathrm{F}^2.
\end{split}
\end{equation}
\end{proof}
%\begin{proof}
%(\ref{easy1}~--~\ref{easy2})~follow directly from Lemma \ref{null}.~Now consider~(\ref{dif1}~~--~~\ref{dif3}).
%
%Denote the~($i, j$)~entry of $B$ by $B_{ij}$.~According to Lemma \ref{null},
%\begin{equation}
%\begin{split}
%&~\int_{\ball_p^r}\Tran{x}Bx\md x\\
%=&~\int_{\ball_p^r}\sum_{i,j}x_iB_{ij}x_j\md x\\
%=&~\int_{\ball_p^r}\sum_{i}B_{ii}x_i^2\md x\\
%=&~I_3\Tr B,
%\end{split}
%\end{equation}
%which justifies~(\ref{dif1}).~(\ref{dif2})~and~(\ref{dif3})~follow as corollaries.~As for~(\ref{diff}),
%\begin{equation}
%\begin{split}
%&~\int_{\ball_p^r}(\Tran{x}Bx)^2\md x\\
%=&~\int_{\ball_p^r}\sum_{i,j,k,l}(x_iB_{ij}x_j)(x_kB_{kl}x_l)\md x\\
%=&~\int_{\ball_p^r}\left(\sum_{i}B^2_{ii}x_i^4+2\sum_{i\neq j}B^2_{ij}x_i^2x_j^2+\sum_{i\neq j}B_{ii}B_{jj}x_i^2x_j^2\right)\md x\\
%=&~I_2\|D\|_\mathrm{F}^2+2I_1(\|B\|^2_\mathrm{F}-\|D\|^2_\mathrm{F})+I_1(\Tr^2B-\|D\|_\mathrm{F}^2)\\
%=&~I_1(2\|B\|_\mathrm{F}^2+\Tr^2B)+(I_2-3I_1)\|D\|_\mathrm{F}^2.
%\end{split}
%\end{equation}
%\end{proof}
\subsection{Proofs}
Now we give the proofs of our main results.

With Lemma \ref{int} and \ref{integrals},~the proof of Theorem \ref{Thp} is nearly completed.~We present it as follows.
\begin{proof}
According to Lemma \ref{integrals},
\begin{equation}
\begin{split}
&~\|Q\|^2_{\Sob{2}{0}(\ball_p^r)}\\
=&~\int_{\ball_p^r}\left[\frac{1}{2}\Tran{x}Bx+\Tran{g}x+c\right]^2\md x\\
=&~\int_{\ball_p^r}\left[\frac{1}{4}(\Tran{x}Bx)^2+(\Tran{g}x)^2+c\Tran{x}Bx+c^2\right]\md x\\
=&~\frac{1}{4}I(2\|B\|_\mathrm{F}^2+\Tr^2B)+\frac{1}{4}(J-3I)\|D\|_\mathrm{F}^2\\
 &~+K(\|g\|_2^2+c\Tr B)+V_{p}r^nc^2,\\
\end{split}
\end{equation}
and
\begin{equation}
\begin{split}
&~|Q|^2_{\Sob{2}{1}(\ball_p^r)}\\
=&~\int_{\ball_p^r}\|Bx+g\|_2^2\md x\\
=&~\int_{\ball_p^r}\left[\Tran{x}B^2x+\|g\|_2^2\right]\md x\\
=&~K\|B\|_\mathrm{F}^2+V_{p}r^n\|g\|_2^2.
\end{split}
\end{equation}
Now apply Lemma \ref{int}.

Since Lemma \ref{int} holds for infinite $p$ in the sense of limit,~so does Theorem \ref{Thp}.
\end{proof}
From Theorem \ref{Thp},~Corollary \ref{Th2} follows directly.
\section{Discussions}\label{diss}
\subsection{The Invariance Under Orthogonal Transformations}
Denote the space of all quadratic functions on $\Real^n$ by $\Qfs$.~A functional $\Fun$ on $\Qfs$ is said to be invariant under orthogonal transformations provided that 
\begin{equation}
\Fun(Q\comps T) = \Fun(Q).
\end{equation}
for any $Q\in\Qfs$ and any orthogonal transformation $T$ on $\Real^n$.

\begin{Prop}Consider $\|\cdot\|_{\Sob{2}{0}(\ball_p^r)}$ and $|\cdot|_{\Sob{2}{1}(\ball_p^r)}$ as functionals on $\Qfs$. Then
\begin{itemize}
\item[a.]$\|\cdot\|_{\Sob{2}{0}(\ball_p^r)}$ is invariant under orthogonal transformations if and only if $p = 2$;
\item[b.]$|\cdot|_{\Sob{2}{1}(\ball_p^r)}$ is invariant under orthogonal transformations for any $p$.
\end{itemize}
\end{Prop}
\begin{proof}
According to Theorem \ref{Thp},~$\|\cdot\|_{\Sob{2}{0}(\ball_p^r)}$ is invariant under orthogonal transformations if and only if 
\begin{equation}\label{Jp3Ip}
%J_p=3I_p,
J=3I,
\end{equation}
since $\|B\|_\mathrm{F}$,~$\Tr B$ and $\|g\|_2$ are invariant while $\|D\|_\mathrm{F}$ is not.~(\ref{Jp3Ip})~is equivalent to 
\begin{equation}\label{Gamma3}
\frac{\Gmf\left(\frac{5}{p}\right)\Gmf\left(\frac{1}{p}\right)}{\left[\Gmf\left(\frac{3}{p}\right)\right]^2} = 3,
\end{equation}
which assumes at most one solution,~since the left hand side of it is strictly decreasing for $p\in(0,\infty)$~according to Proposition \ref{mono}.~Obviously~$p=2$ is a solution to~(\ref{Gamma3}),~and hence the only one.~Thus the invariance of $\|\cdot\|_{\Sob{2}{0}(\ball_p^r)}$ holds if and only if $p=2$.

The invariance of $|\cdot|_{\Sob{2}{1}(\ball_p^r)}$ is easy to check in light of Theorem \ref{Thp}. 
\end{proof}
\subsection{The Sobolev~(Semi)Norms over $\ellp$-Ellipsoids}
By $\ellp$-ellipsoid we mean a set of the form
\begin{equation}
\{x\in\Real^n;~\|A(x-x_0)\|_p\leq r\},
\end{equation}
where $A\in\Real^\nxn$ is a nonsingular matrix.~The Sobolev (semi)norms of $Q$ over an $\ellp$-ellipsoid~can be deduced from Theorem \ref{Thp} easily.
\subsection{The Weighted Sobolev (Semi)Norms}
It may be meaningful to consider weighted Sobolev (semi)norms defined by
\begin{align}
\|Q\|_{\Sob{2}{0}(\ball_p^r,w)}&=\left[\int_{\ball_p^r}w(x)|f(x)|^2\md x\right]^{1/2},\\
|Q|_{\Sob{2}{1}(\ball_p^r,w)}&=\left[\int_{\ball_p^r}w(x)\|\nabla f(x)\|_2^2\md x\right]^{1/2},
\end{align}
where $w$ is a nonnegative bounded function defined on $\ball_p^r$, playing the role of a weight. They can be calculated via the procedure presented in Section \ref{proofs}, and the main labor is integrating related monomials with respect to the weight. 

As an example,~we show without proof $\|Q\|_{\Sob{2}{0}(\ball_2^r,w)}$ and $|Q|_{\Sob{2}{1}(\ball_2^r,w)}$ with the weight
\begin{equation}
w(x) = r^2-\|x\|_2^2.  
\end{equation}
\begin{Prop}\label{Thw}
\begin{align}
%\|Q\|^2_{\Sob{2}{0}(\ball_2^r,w)} &= \frac{2V_{2}\,r^{n+2}}{n+2}\left[\frac{r^4}{4(n+4)(n+6)}\left(2\|B\|_\mathrm{F}^2+\Tr^2 B\right)+\frac{r^2}{n+4}(c\Tr B+\|g\|_2^2)+c^2\right];\\
\|Q\|^2_{\Sob{2}{0}(\ball_2^r,w)} &= \frac{2V_{2}\,r^{n+2}}{n+2}\left[\frac{r^4\left(2\|B\|_\mathrm{F}^2+\Tr^2 B\right)}{4(n+4)(n+6)}+\frac{r^2(c\Tr B+\|g\|_2^2)}{n+4}+c^2\right];\\
|Q|^2_{\Sob{2}{1}(\ball_2^r,w)} &= \frac{2V_{2}\,r^{n+2}}{n+2}\left[\frac{r^2}{n+4}\|B\|_\mathrm{F}^2+\|g\|_2^2\right].
\end{align}
\end{Prop}

%\section{Discussions}\label{dissc}

\appendix

\section*{Appendix} 
\section{Some Propositions about the Gamma Function}
\begin{Prop}\label{limit}
Suppose $\lambda$ is a positive constant.~It holds
\begin{equation}
\lim_{t\rightarrow 0^+}\frac{\Gamma(t)}{\Gamma(\lambda t)} = \lambda.
\end{equation}
\end{Prop}
\begin{proof}
According to the Weierstrass product of Gamma function \cite{artin-gamma},~we have
\begin{equation}
\Gamma(t) = \frac{\me^{-\gamma t}}{t}\prod_{k=1}^{\infty}\left[\left(1+\frac{t}{k}\right)^{-1}\me^{\frac{t}{k}}\right].
\end{equation}
Hence it suffices to show that
\begin{equation}
\lim_{t\rightarrow 0^+}\frac{\prod_{k=1}^{\infty}(1+\frac{t}{k})^{-1}\me^{\frac{t}{k}}}{\prod_{k=1}^{\infty}{(1+\frac{\lambda t}{k})^{-1}\me^{\frac{\lambda t}{k}}}} = 1,
\end{equation}
which is true according to the following deduction.

When $t\in(0,\frac{1}{2(1+|1-\lambda|)})$,
\begin{equation}
\begin{split}
&~\left|\log\frac{\prod_{k=1}^{\infty}(1+\frac{t}{k})^{-1}\me^{\frac{t}{k}}}{\prod_{k=1}^{\infty}{(1+\frac{\lambda t}{k})^{-1}\me^{\frac{\lambda t}{k}}}}\right|\\
=&~\left|\sum_{k=1}^{\infty}\left[\frac{t}{k}-\log\left(1+\frac{t}{k}\right)\right]-\sum_{k=1}^{\infty}\left[\frac{\lambda t}{k}-\log\left(1+\frac{\lambda t}{k}\right)\right]\right|\\
\leq&~\sum_{k=1}^{\infty}\left|\frac{(1-\lambda)t}{k}-\log\left[1+\frac{(1-\lambda)t}{k+\lambda t}\right]\right|\\
=&~\sum_{k=1}^{\infty}\left|\frac{(1-\lambda)t}{k}-\frac{(1-\lambda)t}{k+\lambda t}+\frac{1}{2(1+\xi_k)^2}\cdot\frac{(1-\lambda)^2t^2}{(k+\lambda t)^2}\right|\\
\leq&~\sum_{k=1}^{\infty}\left[\frac{|1-\lambda|\lambda t^2}{k^2}+\frac{(1-\lambda)^2t^2}{2(1/2)^2k^2}\right]\\
=&~\frac{\pi^2}{6}(\lambda+2|1-\lambda|)|1-\lambda|t^2,%\\
\end{split}
\end{equation}
where $\xi_k$ is due to Taylor expansion with Lagrange remainder,~and $\xi_k\geq-\frac{1}{2}$ when $t\in(0,\frac{1}{2(1+|1-\lambda|)})$. 
\end{proof}

\begin{Prop}\label{mono}
Suppose $\lambda$ and $\mu$ are positive constants satisfying $\lambda+\mu\geq 2$.~Then the function
\begin{equation}
\phi(t) = \frac{\Gamma(\lambda t)\Gamma(\mu t)}{[\Gamma(t)]^{\lambda+\mu}}
\end{equation}
is monotonically increasing over $(0,\infty)$,~and the monotonicity is strict unless $\lambda=1=\mu$.
\end{Prop}
\begin{proof}
We suppose $\lambda=1=\mu$ does not happen and prove the strict monotonicity for $\log\phi(t)$~by inspecting its derivative.~To do this,~consider the \emph{digamma} function 
\begin{equation}
\psi(t) = \frac{\md}{\md t}\log\Gamma(t)
\end{equation}
and its partial fraction expansion \cite{artin-gamma} 
\begin{equation}
\psi(t) = -\gamma-\frac{1}{t}+\sum_{k=1}^\infty\frac{t}{k(t+k)}.
\end{equation}
For any $t>0$,
\begin{equation}
\begin{split}
&~\frac{\md}{\md t}\log\phi(t)\\
=&~\lambda \psi(\lambda t)+\mu \psi(\mu t)-(\lambda+\mu)\psi(t)\\
=&~\lambda\left[-\gamma-\frac{1}{\lambda t}+\sum_{k=1}^{\infty}\frac{\lambda t}{k(\lambda t+k)}\right]+\mu\left[-\gamma-\frac{1}{\mu t}+\sum_{k=1}^{\infty}\frac{\mu t}{k(\mu t+k)}\right]\\
&~-(\lambda+\mu)\left[-\gamma-\frac{1}{t}+\sum_{k=1}^{\infty}\frac{t}{k(t+k)}\right]\\
=&~\frac{1}{t}(\lambda+\mu-2)+\sum_{k=1}^{\infty}\frac{t}{k}\left(\frac{\lambda^2}{\lambda t+k}+\frac{\mu^2}{\mu t+k}-\frac{\lambda+\mu}{t+k}\right)\\
=&~\frac{1}{t}(\lambda+\mu-2)+\sum_{k=1}^{\infty}\frac{kt[(\lambda^2+\mu^2)-(\lambda+\mu)]+t^2\lambda\mu(\lambda+\mu-2)}{(\lambda t+k)(\mu t+k)(t+k)}\\
>&~0,
\end{split}
\end{equation}
the last inequality being true because
\begin{equation}
\lambda^2+\mu^2>\lambda+\mu
\end{equation}
due to our assumptions on $\lambda$ and $\mu$.
%Thus the strict monotonicity holds.
\end{proof}

\bibliographystyle{unsrt}
\bibliography{ref}

\end{document}